\documentclass[12pt,a4paper,oneside]{article}
\usepackage[utf8]{inputenc}
\usepackage[T1]{fontenc}
\usepackage{amsthm}

\newtheorem{proposition}{Proposition}

\newtheorem{lemma}{Lemma}
\newtheorem{definition}{Definition}

\newtheorem{remark}{Remark}
\usepackage{subfigure}
\usepackage{grffile}
\usepackage{caption}
\usepackage{latexsym}
\usepackage{graphpap}
\usepackage[english]{babel}
\usepackage{amsfonts}
\usepackage{graphicx}
\usepackage{amsmath,amsfonts,amssymb}

\newtheorem{assump}{Assumption}[section]

\begin{document}
\def\Box{ \framebox[5.5pt]}

\title{\centering \large \bf Decomposition of the wave manifold in Lax admissible regions}

\author{\large \bf C.S.Eschenazi \and \large \bf C.F.B.Palmeira}

\maketitle

\begin{abstract}

Local solutions of Riemann problems for quadratic systems of two conservation laws were constructed in the geometric context. In this paper, also for quadratic systems, we decompose the characteristic and sonic' surfaces in their slow and fast components.These decompositions allow to decompose the wave manifold in regions called admissible region and non admissible region.There are  admissible regions having local shock curve arcs and non local shock curve arcs. Such regions are  important to construct non local  solutions of Riemann problems. Our study is restricted to the symmetric Case IV in the Sheaffer-Shearer classification.

\end{abstract}


\section{Introduction\label{sec:introduction}}
We consider the equation 
\begin{equation}
\left.
W_t+F(W)_x=0,
\label{eq:cons-law}
\right.
\end{equation}
with initial condition
\begin{equation}
W(x,t=0)=\left\{
\begin{array}{ll}
W_L &\hbox{\rm{if }}x < 0\\
W_R &\hbox{\rm{if }}x > 0.
\end{array}
\right.
\label{eq:initial-data}
\end{equation}

\noindent We will take $W=(u,v)\in \mathbb {R}^2$and $F$ a map $F:\mathbb{R}^2\longrightarrow \mathbb{R}^2$. Equation \eqref{eq:cons-law} appears in fluid dynamic, and together with initial conditions \eqref{eq:initial-data} is called a Riemann problem. 

We are interested in the so called shock solutions, defined by
 
\begin{equation*}
W=\left\{
\begin{array}{ll}
W_1 &\hbox{\rm{for }}x < st\\
W_2 &\hbox{\rm{for }}x > st.
\end{array}
\right.
\end{equation*}

\noindent where $s$ is the speed of shock propagation, $W_1=(u_1,v_1)$ and $W_2=(u_2,v_2)$ are the states to be connected by the shock.

In order to have physical meaning, the speed and the states must satisfy the so called {\it Rankine-Hugoniot} condition 

$$F(W_1)-F(W_2)=s(W_1-W_2),$$ 

\noindent for some $s$. This leads to the definition of Hugoniot curve associated to a given state $(u_0,v_0)$, in section \ref{sec:wmhc}.

Furthermore, not all arcs of Hugoniot curves are useful to construct solutions. Shock curve arcs must satisfy some extra conditions called admissibility conditions. In \cite{AEMP10} Liu's admissibility entropy criterion was introduced within the wave manifold context. There it was shown that under certain extra assumptions, which we consider in the current paper, Liu's entropy criterion can be replaced by the Lax's inequalities conditions. These inequalities relate $s$ and the eigenvalues of $DF(u,v)$. Hugoniot arc curves satisfying these inequalities are called {\it shock curve arcs or admissible arcs}.

We adopt here the topological point of view, as described in \cite{Isaacson92}. We consider the space $\mathbb{R}^5$ of coordinates $(u_0,v_0,u,v,s)$ and in it the 3 dimensional manifold defined by $F(u,v)-F(u_0,v_0)-s(u-u_0,v-v_0)=0$. In this manifold, we consider the curves defined by $(u_0,v_0)$ constant, called Hugoniot curves, which project in the $(u,v)$ plane on the classicals Rankine-Hugoniot curves,\cite{AEMP10}. This manifold is called {\it wave manifold}.      

In a series of papers, this manifold and its Hugoniot curves have been studied in the case where $F$ is a polynomial of degree two. The wave manifold has been characterized, relevant surfaces (characteristic, sonic and sonic') have been defined, the intersection of Hugoniot curves with these surfaces has been studied, and the Lax inequalities have been interpreted in this context. Here, also considering $F$ as a polynomial of degree two, we decompose the characteristic and sonic' surfaces in their fast and slow components and also decompose the wave manifold in regions which we call admissible region or non admissible region. 

As stated in the beginning of Section \ref{sec:wmhc}, we will restrict ourselves to the symmetric Case IV, in the Sheaffer-Shearer classification. 

This paper is organized as follows: In Section \ref{sec:wmhc} we review some basic facts and definitions, introduce new variables and describe the characteristic, sonic and sonic' surfaces, \cite{Isaacson92}. In Section \ref{sec:cscf} we characterize the slow and fast components of the characteristic surface associated with the eigenvalues of $DF$, we also characterize the fold curve, which is the boundary of these two components. In Section \ref{sec:waveman} we describe how the characteristic, sonic and sonic' surfaces divide the wave manifold into twelve regions and characterize the surface formed by the Hugoniot curves through points of the fold curve, this surface is tangent to the characteristic and sonic' surfaces. In Section \ref{sec:sonlifs} we characterize the slow and fast components of the sonic' surface, associated with the slow shock speed and the fast shock speed, we show that the boundaries of these two components are a straight line and the sonic' fold curve, which is the  curve where Hugoniot curves are tangent to the sonic' surface. In Section \ref{sec:lax} we identify some of regions in the wave manifold where Lax's inequalities are satisfied, indicating in which regions there are local shock curve arcs and in which of these regions  there are non local shock curve arcs. It is well known, \cite{Isaacson92}, that a Hugoniot curve through points of the secondary bifurcation has two components, a straight line and a curve. The lines generate a plane, and the curves generate a surface, which we will call {\it Sigma}, and construct in Appendix A. In Appendix B we characterize the regions of the wave manifold where condition L3, introduced in section \ref{sec:lax}, is satisfied.

We hope that this paper can be completed in the future with a full characterization of all regions where Lax conditions are satisfied.

\section {\bf The wave manifold and Hugoniot curves\label{sec:wmhc}}

\vspace{.5cm}


In this section we review some  basic facts and definitions and introduce new variables. We will define the characteristic, sonic and sonic' surfaces, since they will appear as boundaries of the admissible regions.

\vspace{.5cm}
We refer the reader to section 2 of \cite{Eschenazi02} for basic definitions. For sake of completeness we will briefly present what is needed for this paper. We consider the equation (\ref{eq:cons-law}), with $F$ given by $F=(f,g)$, and

\begin{equation}
\left\{
\begin{array}{l}
f(u,v)=v^2/2+(b_1+1)u^2/2+a_1u+a_2v\\
\\
g(u,v)=uv+a_3u+a_4v
\end{array}
\right.
\end{equation}
where, $b_1>1$. This is the symmetric case IV  in the classification of Schaeffer and Shearer, \cite{Schaeffer87}.

Given a point $W=(u,v)$ in what we call the state space, {\it the Hugoniot curve through this point is defined as the set of points $W'=(u',v')$ such that there exists $s$ such that $F(W)-F(W')=s(W-W')$}. It is clearly a curve passing in $W$.

To study Hugoniot curves we consider $\mathbb{R}^5=\{(u,v,u',v',s)\}$ and in it the 3-dimensional manifold defined by $F(W)-F(W')=s(W-W')$. Since this manifold is singular along the diagonal $W=W'$, we perform a blow up, along this diagonal, which in this simple case is obtained by changing coordinates to $U=\frac{u+u'}{2}$, $V=\frac{v+v'}{2}$, $X=u-u'$, $Y=v-v'$, $Z=\frac{Y}{X}$ and factoring $X^2$. We also set $c=a_3-a_2>0$. Using these new coordinates, we get 
$$(1-Z^2)(V+a_2)-Z(b_1U+a_1-a_4)+c=0$$ and $$Y=ZX.$$

Since $Z$ is a direction, we may think of the wave manifold $M$, as contained in $\mathbb R^3 \times \mathbb RP^1$, or, which is the same, $\mathbb R^3\times S^1$. These coordinates are not valid at $Z=\infty$, but there are no special features at infinity, so we can just use $X$, $U$, $V$ to study Hugoniot curves.

Defining $\tilde{U}=b_1U+a_1-a_4$ and $\tilde{V}=V+a_2$, one gets the equation $G_1=(1-Z^2)\tilde V-Z\tilde U+c$, used in \cite{Marchesin94b}, \cite{Eschenazi02}, \cite{Eschenazi13}.

Hugoniot curves in the wave manifolds are defined by $u=constant$ and $v=constant$, which, in these new coordinates, become 
\begin{equation}
\left \{
\begin{array}{l}
2\tilde{U}+b_1X=k\\
2\tilde{V}+ZX=l 
\end{array}
\right.
\label{eq:KL}
\end{equation}
\noindent
As shown in \cite{Marchesin94b}, Hugoniot curves are connected and foliate $M^3$ except along a straight line $B$, contained in the plane, $\Pi$, of equation $Z=0$. The straight line $B$ is called {\it secondary bifurcation}. Hugoniot curves through points in $B$ are formed by two arcs: a straight line $hugl$ and a curve $hugc$. These two arcs intersect transversely in $B$. The straight lines $hugl$ fill the plane $\Pi$, and the curves $hugc$ form a surface $\Sigma$.

The same conclusions are valid for Hugoniot' curves, defined by $u'=constant$ and $v'=constant$. Hugoniot' curves are connected and foliate $M^3$ except along a straight line $B'$, also contained in the plane $\Pi$. Hugoniot' curves through points in $B'$ are formed by two arcs: a straight line $hugl'$ and a curve $hugc'$. These two arcs intersect transversely in $B'$. The straight lines $hugl'$ also fill the plane $\Pi$, and the curves $hugc'$ form a surface $\Sigma'$. 

Three surfaces in the wave manifold are important in this study: Characteristic ($\mathcal {C}$), Sonic ($Son$) and Sonic' ($Son'$).

To define $Son$ and $Son'$, one must look at the speed $s$ as a real function in $M^3$ given by $\displaystyle\frac{f(u,v)-f(u',v')}{u-u'}$ or $\displaystyle\frac{g(u,v)-g(u',v')}{v-v'}$. Using coordinates $X$, $\tilde{U}$ and $\tilde{V}$, we have 
  
\begin{equation}
\left.  
s=\frac{\tilde U}{b_1}+ \frac{ V}{Z},
\right.
\label{eq:sp1}
\end{equation}

To define $Son$, we restrict $s$ to a Hugoniot curve, and look at its critical points. The set of these critical points for all Hugoniot curves is the Sonic surface, $Son$. In the same way we define $Son'$, as the set of all critical points of $s$ restricted to a Hugoniot' curve. This was done in \cite{Marchesin94b}, obtaining equations
\begin {equation}
\left.
2(Z^2+b_1+1)\tilde{U}+2(Z^3+(b_1+3)Z)\tilde{V}-(Z^2-(b_1+1))X=0
\right.
\label{eq:eqson}
\end{equation}

\noindent for $Son$ and 

\begin{equation}
\left.
2(Z^2+b_1+1)\tilde{U}+2(Z^3+(b_1+3)Z)\tilde{V}+(Z^2-(b_1+1))X=0 
\right.
\label{eq:eqson'}
\end{equation}

\noindent for $Son'$.

We list some important known facts about $\mathcal{C}$, $Son$ and $Son'$.

\begin{itemize}
\item[Fact 1-] $\mathcal{C}$ is topologically a cylinder;

\noindent \item[ Fact 2-] Given a Hugoniot curve, either it intersects $\mathcal{C}$ transversally in two points, or it is tangent or it does not intersect $\mathcal{C}$.

\noindent \item[ Fact 3-] The above mentioned set of tangencies points is a circle, which divides $\mathcal{C}$ into two components, denoted by $\mathcal{C}_s$ and $\mathcal{C}_f$. This circle is called {\it coincidence curve} or {\it fold curve}. 

\noindent \item[Fact 4-] Any Hugoniot curve through a point $\mathcal{U}$, $sh(\mathcal{U})$, which intersects $\mathcal{C}$ transversally, does it in a point $\mathcal{U}_s=sh(\mathcal{U})\cap \mathcal{C}_s$ and in another point $\mathcal{U}_f=sh(\mathcal{U})\cap \mathcal{C}_f$. The value of the speed $s$ in $\mathcal{U}_s$ is smaller than the value of $s$ in $\mathcal{U}_f$. Proofs of these facts will be in section \ref{sec:cscf} where we will characterize $\mathcal{C}_s$ and $\mathcal{C}_f$.

\noindent \item[ Fact 5-] $Son$ and $Son'$ are topologically Moebius band and they intersect $\mathcal {C}$ along a curve called {\it inflection locus}, which has two connected components. They also intersect along two straight lines, each of which intersects transversally $\mathcal {C}$ in a point of the inflection locus.

\noindent \item[ Fact 6-] Generically, a Hugoniot curve through a point $\mathcal {U}$, $sh(\mathcal {U})$, intersects $Son'$ in 0, 2 or 4 points. If $sh(\mathcal {U}) \cap Son'$ has 2 points, then $s$ has the same  value in these 2 points as the value of $s$ either in $\mathcal {U}_s$ or in $\mathcal {U}_f$. In the first case the points are said to be in $Son'_s$ and in  the second case in $Son'_f$. If the intersection has 4 points, then $s$ has the same value of $s$ in $\mathcal {U}_s$ in two of them and the value of $s$ in $\mathcal {U}_f$ in the other two. The first two are points in $Son'_s$ and the other two in $Son'_f$.

\end{itemize}
 
\subsection{\bf New coordinates \label{sec:Yzt-subsection}}

In this paper we use coordinates $z=\displaystyle \frac{1}{Z}$, $Y$ and $t$. The main advantage is that they are valid for $M^3-\Pi$,  and $\Pi$ becomes the plane $z=\infty$. To define $t$, we start by writing the equation of $M$ in $\tilde{U}$, $V_1=V+a_3=\tilde{V}+c$ and $z$. We get $G=0$, where $G=(z^2-1)V_1-z\tilde{U}+c$. The parameter $t$ is defined by

\begin{equation}
\tilde U=\displaystyle \frac{2cz}{z^2+1}+t(z^2-1); \\
 V_1=\displaystyle \frac{c}{z^2+1}+tz.
\label{eq:tdef}
\end{equation}

\noindent  For each fixed $z$ parameter $t$ measures as far as $\tilde U$ and $V_1$ are away from the fold curve in the orthogonal direction to it.

In these new coordinates, the characteristic surface, $\mathcal C$,  is the plane $Y=0$, the sonic surface and the sonic' surface are given by the equations obtained from equations \eqref{eq:eqson} and \eqref{eq:eqson'} by changing $\tilde{U}$, $\tilde{V}$, $X$ and $Z$ by $Y$, $t$ and $z$. We get $son=0$, where

\begin{equation}
son= -2c((b_1+1)z^5+(b_1+4)z^3+3z)t-((b_1+1)z^4+b_1z^2-1)Y-2c((b_1-1)z^2+1),
\label{eq:son}
\end{equation}

\noindent and $son'=0$, where
\begin{equation}
son'=-2c((b_1+1)z^5+(b_1+4)z^3+3z)t+((b_1+1)z^4+b_1z^2-1)Y-2c((b_1-1)z^2+1).
\label{eq:son'}
\end{equation}

The shock  speed, $s$,  is written as
\begin{equation}
s=\displaystyle\frac{  c((b_1+1)tz^4+b_1z^2t+(b_1+2)z-t)}{b_1(z^2+1)}
\label{eq:speed}
\end{equation}

Parametric equations for Hugoniot curves in $Y$, $t$, $z$ coordinates are obtained by replacing $\tilde{U}$, $\tilde{V}$ and $X$ into equations in (\ref{eq:KL}), by their values in terms of $Y$, $t$, $z$ and solving the system with respect to $Y$ and $t$. We get:

\begin{equation}
\left\{
\begin{array}{l}
Y=\displaystyle -\frac{(l+2c)z^2-kz-l}{(b_1-1)z^2+1}\\
t=\displaystyle \frac{(z^2+1)(b_1zl-k)+2cz(2+b_1z^2)}{2c((b_1-1)z^4+b_1z^2+1)}.
\end{array}
\right.
\label{eq:parahug}
\end{equation}

Parametric equations for Hugoniot' curves are obtained from equations (\ref{eq:parahug}) by changing $Y$ by $-Y$. 

Substituting the second equation of (\ref{eq:parahug}) into equation (\ref{eq:speed}) we get the expression of the shock speed $s=s(k,l,z)$ along a Hugoniot curve.

In order to get the expression of a Hugoniot curve through a given point $(t_0,z_0,Y_0)$, we start by solving system  \eqref{eq:parahug} in $k$ and $l$ obtaining\\ $k(t,z,Y)=\frac{4cz+2ct(z^4-1)+b_1z(z^2+1)Y}{z^2+1}$ and $l(t,z,Y)=\frac{2cz(z^2+1)t+Y(z^2+1)-2cz^2}{z^2+1}$.

Substituting $k((t_0,z_0,Y_0)$ and $l((t_0,z_0,Y_0)$ into the equations in \eqref{eq:parahug}, we get the parametric equations for the Hugoniot curve through a point $(t_0,z_0,Y_0)$, 

\begin{equation}
\left\{
\begin{array}{l}
Y=\displaystyle \frac{Az^2+Bz +C}{(z_0^2+1)((b_1-1)z^2+1)}\\
t=\displaystyle \frac{Dz^3+Ez^2+Fz+G}{c(z_0^2+1)((b_1-1)z^4+b_1z^2+1)},
\end{array}
\right.
\label{eq:hugponto}
\end{equation}

\noindent where\\
$A=-(2c+2ct_0z_0(1+z_0^2)+Y_0(z_0^2+1));$\\
$B=4cz_0+2ct_0(z_0^4-1)+b_1Y_0z_0(z_0^2+1);$\\
$C=-2cz_0^2+2ct_0z_0(1+z_0^2)+Y_0(z_0^2+1)$;\\
$D=2cb_1+2cb_1t_0z_0(z_0^2+1)+b_1Y_0(z_0^2+1)$;\\
$E=2ct_0(1-z_0^4)-b_1z_0Y_0(z_0^2+1)-4cz_0$;\\
$F=4c(z_0^2+1)+2cb_1t_0z_0(1+z_0^2)+b_1Y_0(z_0^2+1)-2cb_1z_0^2$;\\
$G=-4cz_0+2ct_0(1-z_0^4)-b_1z_0Y_0(1+z_0^2)$.

\section {\bf Characterizing $\mathcal{C}_s $ and $\mathcal{C}_f$\label{sec:cscf}}

\vspace{.5cm}

As we will see in section \ref{sec:lax} admissible shock curve arcs will either start in $\mathcal{C}_s$ or in $Son'_s$. So we must characterize $\mathcal{C}_s$ and $\mathcal{C}_f$. In section \ref{sec:sonlifs} we will define $Son'_s$ and $Son'_f$ and characterize them.

The speed $s$ along a Hugoniot curve through a point $(t_0,z_0,Y_0)$ is obtained by substituting $k(t_0,z_0,Y_0)$ and $l(t_0,z_0,Y_0)$ into the expression $s(k,l,z)$, giving 

\begin{equation}
s_{hug}= \frac{s_{p3}z^3-s_{p2}z^2-s_{p1}z+s_{p0}}{2b_1(z_0^2+1)((b_1-1)z^2+1)}
\label{eq:velhug}
\end{equation}
\noindent where\\
$s_{p3}=b_1(b_1+1)(Y_0(z_0^2+1)+2c(1+z_0t_0(z_0^2+1)));$\\
$s_{p2}=(b_1+1)(b_1z_0Y_0(z_0^2+1)+2c(2z_0+t_0z_0^4-t_0));$\\
$s_{p1}=b_1(Y_0(z_0^2+1)+2c(z_0t_0(z_0^2+1)-2z_0^2-1));$\\
$s_{p0}=b_1z_0Y_0(z^2+1)+2c(2z_0+t_0z_0^4-t_0)$.

The equation of the shock speed $s$ along a Hugoniot curve through a point $(t_0,z_0,0)$ in the characteristic surface ( characteristic speed) is obtained by substituting $Y_0=0$ into equation \eqref{eq:velhug}

\begin{equation}
s_{ch}=\displaystyle \frac{s_{ch3}z^3- s_{ch2}z^2- s_{ch1} z+s_{ch0}}{b_1(z_0^2+1)((b_1-1)z^2+1)},
\label{eq:velcar}
\end{equation}

\noindent where\\
$s_{ch3}=b_1c(t_0z_0(1+z_0^2)+1)(1+b_1)$;\\
$s_{ch2}=c(t_0z_0^4 +2z_0-t_0)(1+b_1)$;\\
$s_{ch1}=b_1c(t_0z_0^3-2z_0^2+t_0z_0-1)$;\\
$s_{ch0}=c(t_0z_0^4 +2z_0-t_0)$.

\begin{lemma}
In these coordinates, the fold curve is given by $Y=0$ and $t=0$. $\mathcal{C}_s$ is given by $t<0$ and $\mathcal{C}_f$ by $t>0$.
\label{lemma:cf-cs}
\end{lemma}
\begin{proof}
A Hugoniot curve through a point on the characteristic  has parametric equations
\begin{equation}
\left\{
\begin{array}{l}
Y=\displaystyle \frac{-2c(z-z_0)(t_0z_0(1+z_0^2)+1)z+(t_0(1+z_0^2)-z_0)}{(z_0^2+1)((b_1-1)z^2+1)}\\
t=\displaystyle \frac{D_0z^3+E_0z^2+F_0z+G_0}{c(z_0^2+1)((b_1-1)z^4+b_1z^2+1)},
\end{array}
\right.
\label{eq:hugch}
\end{equation}
\noindent where $D_0$, $E_0$, $F_0$ and $G_0$ are obtained from (\ref{eq:hugponto}). It is just a matter of making $Y_0=0$ in the equations (\ref{eq:hugponto}). 

Solving equation $Y=0$ we get  $z_0$ and $ z_1=-\displaystyle \frac{t_0(1+z_0^2)-z_0}{t_0z_0(1+z_0^2)+1},$  values of  $z$ in intersection points with $C$. It is easy to see that $z_0=z_1$ if and only if $t=0$, characterizing thus the fold curve. 

Let $s_{ch0}$ and $s_{ch1}$  be the characteristic speeds for $z=z_0$ and $z=z_1$. Straightforward computations give $s_{ch0}-s_{ch1}=c(z_0^2+1)t_0$. So $ s_{ch0}>s_{ch1}$ if and only if $t_0>0$. In this way $C_f$ is the half plane $t>0$ and $C_s$ is the half plane $t<0$.  
\end{proof}

\section {\bf Decomposition of the Wave Manifold\label{sec:waveman}}

In this section we describe how the characteristic, sonic and sonic' surfaces divide $M^3-\Pi$.

It follows from equations \eqref{eq:son'} and \eqref{eq:son} that $\mathcal C$, $Son$ and $Son'$ are ruled surfaces. For each fixed $z\neq 0$ we have three straight lines in the $tY$-plane. The intersection of this plane with $\mathcal C$ is just the horizontal axis $Y=0$. In the same way, its intersections with $Son$ and $Son'$ are lines which we will call $Sz$ and $S'z$, intersecting along $Y=0$. In $\mathcal C$ these are points in the inflection locus, $IL$. For $z=0$ the straight lines become horizontal with equations $(t,z=0,Y=2c)$ for $Son$ and $(t,z=0,Y=-2c)$ for $Son'$.

Besides the inflection locus, $Son$ and $Son'$ intersect for $z=\pm \displaystyle \frac {1}{\sqrt{b_1+1}}$, values of $z$ that cancel out the coefficient of $Y$ in equations \eqref{eq:son'} and \eqref{eq:son}. Straihgtforward computations show that for $z= \displaystyle \frac {1}{\sqrt{b_1+1}}$, $t=- \displaystyle \frac{b_1 \sqrt{b_1+1}}{2(b_1+2)}$ and for $z=- \displaystyle \frac {1}{\sqrt{b_1+1})}$, $t= \displaystyle \frac{b_1 \sqrt{b_1+1}}{2(b_1+2)}$. So, for these values of $z$ the two straight lines are coincident and become vertical, $z$ and $t$ are fixed and $Y$ is any real. These two vertical lines are called {\it double sonic locus}. In this way,  $Son \cap Son'$ is formed by the  inflection locus and  the double sonic locus. We can state.

\begin{proposition}
The characteristic, sonic and sonic' surfaces divide $M^3-\Pi$ into twelve regions. 
\label{proposition:divideM}

\end{proposition}
\begin{proof}
Due to symmetry in $Y$, it is enough to consider the half-space $Y>0$, since a symmetric division will appear in the $Y<0$ half-space. For each fixed $z$ we will look at the 2-dimensional regions in which the plane $\Pi_z$, defined by fixing $z$, is divided and see how these 2-dimensional regions form 3-dimensional regions as $z$ moves. As described above, there are 3 critical values in the $z$ 
movement: $z=-\displaystyle \frac {1}{\sqrt{b_1+1}}$, $z=0$ and $z=\displaystyle \frac {1}{\sqrt{b_1+1}}$.  As $z$ moves, $Sz$ and $S'z$ move and generate a 3-dimensional region in space. Let us describe what happens when $z$ crosses each critical value:

We start with $z=-\displaystyle \frac {1}{\sqrt{b_1+1}}$. As $z$ approaches this value, $Sz$ and $S'z$ become vertical, so the region between $Son$ and $Son'$ for $z<-\displaystyle \frac {1}{\sqrt{b_1+1}}$ disappears. When $z$ crosses the critical value, $Sz$ and $S'z$ just change their relative positions and the two regions bounded, respectively by $Son$ and $\mathcal C$ and $Son'$ and $\mathcal C$ connect with the corresponding region in $-\displaystyle \frac {1}{\sqrt{b_1+1}}<z<0$. So in the half space defined by $z<0$ and $Y>0$ there are four regions: one bounded by $Son$ and $Son'$ contained in $z<-\displaystyle \frac {1}{\sqrt{b_1+1}}$; one bounded also by $Son$ and $Son'$ contained in $-\displaystyle \frac {1}{\sqrt{b_1+1}}<z<0$; and two bounded by 
$Son$, $Son'$ and $\mathcal C$ contained in $z<0$. Symmetrically there are four regions in $z>0$.
Let us describe how the three regions in $\displaystyle \frac {1}{\sqrt{b_1+1}}<z<0$ connect with the three regions in $0<z<\displaystyle \frac {1}{\sqrt{b_1+1}}$. 

As $z \to 0^+$, the intersection point of $Sz$ and $S'z$ goes to $-\infty$, $Sz$ becomes the line $Y=-2C$ and $S'z$ becomes the line $Y=-2c$. So, the region under $S'z$ is pushed to infinity and does not connect to the regions on $z<0$, generating two regions: one above $Sz$ and one limited by $Sz$ and $\mathcal C$. In this way we have six regions in the half space $Y>0$: two regions for $-\displaystyle \frac{1}{\sqrt{b_1+1}}<z<\displaystyle \frac {1}{\sqrt{b_1+1}}$; 2 regions for $z<-\displaystyle \frac {1}{\sqrt{b_1+1}}$ and two regions for $z>\displaystyle \frac {1}{\sqrt{b_1+1}}$. Symmetrically there are six more regions in the half-space $Y<0$.
\end{proof}

\vspace{.5cm}


Figure 1 illustrates the twelve regions in $M^3$. Sonic and sonic' surfaces intersect along the inflection locus (the hyperbola like curve in the characteristic surface) and also along two vertical lines transversal to the characteristic. We call {\it $SS'$} the  3-dimensional regions contained in $z^2> \displaystyle \frac {1}{b_1+1}$ and bounded by $Son$, $Son'$, and $IL$. There are four such regions, which we can specify by the signs of $z$ and $Y$. We call {\it lateral regions} the regions bounded by $Son$, $\mathcal C$ and $Son'$ contained, respectively, in $z>0$ and in $z<0$. There are four such regions, again we can specify them by the signs of $z$ and $Y$. Let us describe the four regions in $z^2< \displaystyle \frac {1}{b_1+1}$. In subspace $Y>0$, the sonic surface is connected, we call it {\it bridge}, and the sonic' surface is not connected. We have two regions: the {\it above bridge} region bounded by $Son$ and $Son'$ and the {\it below bridge} region bounded by $\mathcal C$, $Son$ and $Son'$. In subspace $Y<0$, the sonic' surface is connected, we call it {\it tunnel}, and the sonic surface is not connected. We have two regions: one bounded by $Son$, $Son'$ and $\mathcal C$, called {\it above tunnel} and other bounded by $Son$ and $Son'$, called {\it below tunnel}.
\begin{figure}[htpb]
\begin{center}
\includegraphics[scale=0.8,width=0.5\linewidth]{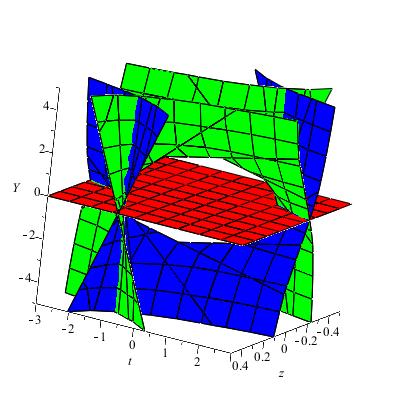}
\caption[]{A view of $C$, $Son$ and $Son'$ in $M^3$}
\end{center}
\end{figure}

We recall that the fold curve in $\mathcal C$ separates $C_f$ and $C_s$. Hugoniot curves through points of the fold curve are tangent to $\mathcal C$ and $Son'$, \cite{Eschenazi02}. We will denote by $Tf$ the 2-dimensional submanifold of $M^3$ generated by Hugoniot curves through points of the fold curve.

\begin{lemma}
The surface $Tf$ is tangent to $\mathcal{C}$ and to $Son'$.
\label{lemma:Tf}
\end{lemma}
\begin{proof}
A Hugoniot curve through a point of the fold curve is given by\\ $2\tilde{U}+b_1zY=k$, $2(V-c)+Y=l$, where $(k,l)$ is a point of the ellipse\\ $k^2+4l^2+8cl=0$. Using equations \eqref{eq:tdef} in the expressions of $k$ and $l$ and substituting in the ellipse equation, we get equation for the surface $Tf$:
\begin{equation}
\frac{tf_1Y^2+tf_2Y+tf_3}{z^2+1}=0
\label{eq:tfequ}
\end{equation}

\noindent where

\noindent $tf_1=(z^2+1)(b_1^2z^2+4)$;\\
$tf_2=4c(z(z^2+1)(b_1z^2-b_1+4))t+2((b_1-1)z^2+1)$;\\
$tf_3=4c^2t^2(z^2+1)^3.$

The intersection of $Tf$ and $\mathcal {C}$ is obtained putting $Y=0$ in equation \eqref{eq:tfequ}. Doing so we get $4c^2t^2(z^2+1)^2=0$. It follows that $Tf$ is tangent to $\mathcal {C}$ along the straight line $t=0$, the fold curve.

Hugoniot curves are tangent to $Son'$ along a curve called {\emph{sonic' fold}}, somewhat improperly called sonic fold in \cite{Eschenazi02}.
To show that $Tf$ surface is tangent to $Son'$ we solve the equation of \eqref{eq:son'} in $Y$, getting
\begin{equation}
\textstyle Y=\textstyle \frac{2c((b_1+1)z^5+(b_1+4)tz^3+(b_1-1)z^2+3tz+1)}{(b_1+1)z^4+b_1z^2-1}
\label{eq:sonlizt}
\end{equation}

By Substituting $Y$ from equation \eqref{eq:sonlizt} in equation \eqref{eq:tfequ} we obtain
\begin{equation}
\left(
\textstyle \frac{2c((1+z^2)((b_1+1)^2z^4+2(b_1+3)z^2+1)t+(2+b_1)z((b+1-1)z^2+1)}{(1+z^2)((b_1+1)z^2-1)}
\right)^2=0.
\label{eq:Tfzt}
\end{equation} 

So $Tf$ is tangent to $Son'$. $Tf$ is topologically a cylinder, each section $z=constant$ is an ellipse. It is easy to see that all ellipses are tangent to the $t$ axis at $(0,0)$ and contained in the $Y<0$ half plane, so $Tf$ is contained in the $Y<0$ half space, and in fact, in the above tunnel region.
\end{proof} 

We remark that equation \eqref{eq:Tfzt} is the projection of $Son'\cap Tf$ (sonic' fold curve) in plane $(z,t)$. 

Parametric equations for sonic' fold curve are obtained solving  equation \eqref{eq:Tfzt} in $t$, and substituting $t(z)$ in equation \eqref{eq:sonlizt}. Straightforward calculations give the parametric equations as,
\begin{equation}
\left \{
\begin{array}{l}
z=z \\
t=-\frac{(b_1+2)z((b_1-1)z^2+1))}{(z^2+1)((b_1+1)^2z^4+2(b_1+3)z^2+1)}\\  
Y=-\frac{2c((b_1-1)z^2+1)}{(b_1+1)^2z^4+2(b_1+3)z^2+1},
\end{array}
\right.
\label{eq:parsonli}
\end{equation}


Changing $Y$ by $-Y$ in equation \eqref{eq:tfequ} we get equation of $Tf'$, the surface generated by Hugoniot' curves through points of the fold curve. The surface $Tf'$ is just the reflection of $Tf$ in the $Y=0$ plane, so it is contained in the $Y>0$ half space under the bridge.

In a similar way as in lemma \ref{lemma:Tf}, the surface $Tf'$ is tangent to $\mathcal C$ and to $Son$.

\section {\bf Finding $Son'_s$ and $Son'_f$. \label{sec:sonlifs}}

In this section we describe how $Son'$ splits into {\it slow sonic' surface}, $Son'_s$, and {\it fast sonic' surface}, $Son'_f$. In order to do so we will take a point in $Son'$, the Hugoniot curve through this point, find the intersection with $\mathcal {C}$, calculate the shock speed in each of these 2 points and see which point of $\mathcal C$ has the same speed as our initial point in $Son'$.

Given a point in $Son'$, $(t'_0,z_0,Y_0)$, $t'_0= \frac{Y_0(z_0^2+1)((b_1+1)z_0^2-1)-2c((b_1-1)z_0^2+1)}{cz_0(z_0^2+1)((b_1+1)z_0^2+3)}$. 

\noindent The shock speed along a Hugoniot curve through this point is obtained by substituting $t_0$ by $t'_0$ into equation \eqref{eq:velhug}
\begin{equation}
s_{son'}=\frac{s_{son'3}z^3-s_{son'2}z^2-s_{son'1}z+s_{son'0}}{2b_1z_0((b_1-1)z^2+1)((b_1+1)z_0^2+3))}
\label{eq:speedson'}
\end{equation}

\noindent where,\\
$s_{son'3}=2b_1z_0(b_1+1)(Y_0(b_1+1)z_0^2+1)+2c);$\\
$s_{son'2}=(b_1+1)(Y_0((b_1z_0^2+1)^2+2z_0^2(b_1z_0^2-1)+z_0^4)+2c((b_1+3)z_0^2+1));$\\
$s_{son'1}=2b_1z_0(Y_0(b_1+1z_0^2+1)-2c(b_1+1)z_0^2+2));$\\
$s_{son'0}=(Y_0((b_1z_0^2+1)^2+2z_0^2(b_1z_0^2-1)+z_0^4)+2c((b_1+3)z_0^2+1)).$

\noindent The shock speed at this point, $s_{son'}(t_0',z_0,Y_0)$, is obtained putting $z=z_0$ into equation \eqref{eq:speedson'},

\begin{equation}
s_{son'}(t_0',z_0,Y_0)=\frac{((b_1+1)z_0^2-1)^2Y_0+6c(b_1+1)z_0^2+2c}{2b_1z_0((b_1+1)z_0^2+3)}
\label{eq:speedsonponto'}
\end{equation} 

Parametric equations for the Hugoniot curve through $(t_0',z_0,Y_0)$ are obtained changing $t_0$ into $t_0'$ in equations \eqref{eq:hugponto}, giving $(t_{hu}(z),z,Y_{hu}(z))$. Solving equation $Y_{hu}=0$ we obtain the $z$ coordinates of  the intersection points of Hugoniot curve with $\mathcal C$, $z_{\mathcal {C}_1}= \frac{(b_1+1)z_0^2+1}{2z_0}$, $z_{\mathcal {C}_2}= -\frac{2(Y_0-c)z_0}{(b_1+1)Y_0z_0^2+Y_0+2c}$. Substituting $z_{\mathcal {C}_1}$ and $z_{\mathcal {C}_2}$ in the expression of $t_{hu}(z)$ we get the $t$ coordinates of intersection points, 
$t_{\mathcal {C}_1}=\frac{((b_1+1)^2z_0^4+2(b_1+3)z_0^2-1)Y_0-2c((b_1+3)z_0^2+1)}{c((b_1+1)z_0^2+3)((b_1+1)^2z_0^4+2(b_1+3)z_0^2+1)}$\\ and\\
$t_{\mathcal {C}_2}=-\frac{((b_1+1)z_0^4+2(b_1+3)z_0^2+1)Y_0+2c((b_1-1)z_0^2+1))(((b_1+1)z_0^2+1)Y_0+2c)^2}{2cz_0((b_1+1)z_0^2+3))((b_1+1z_0^4+2(b_1+3)z_0^2+1)Y_0^2+4c((b_1-1)z_0^2+1)Y_0+4c^2(z_0^2+1))}$, res\-pec\-ti\-ve\-ly. So intersection points of Hugoniot curve with $\mathcal C$ are $(t_{\mathcal {C}_1}, z_{\mathcal {C}_1},0)$ and $(t_{\mathcal {C}_2}, z_{\mathcal {C}_2},0)$. 

Substituting in equation \eqref{eq:speed} we obtain \\$s_{\mathcal {C}_1}=\frac{((b_1+1)^3z_0^4+2((b_1+1)^2-2)z_0^2+b_1+1)Y_0+2c((b_1+1)^2z_0^2+2z_0^2+b_1+1)}{2z_0b_1((b_1+1)z_0^2+3)}$ and\\
$s_{\mathcal {C}_2}=\frac{((b_1+1)z_0^2-1)^2Y_0+6c(b_1+1)z_0^2+2c}{2b_1z_0((b_1+1)z_0^2+3)}$. A simple inspection shows that $s_{\mathcal {C}_2}=s_{son'}$. 

According to Lemma 1, we must to study the signal of $t_{\mathcal {C}_2}$. If $t_{\mathcal {C}_2}>0$, the point $(t_0,z_0,Y_0)$ is in $Son'_f$, otherwise is in $Son'_s$.

\begin{proposition}  
The sonic' fold curve and the straight line $(t,z=0,Y=-2c)$ are the boundaries of $Son'_s$ and $Son'_f$.
\label{proposition:son'sson'f}
\end{proposition}

\begin{proof} We have to study the sign of $t_{\mathcal {C}_2}$. The denominator is the product of $2cz_0((b_1+1)z_0^2+3))$ and a polynomial of degree 2 in $Y_0$ with negative discriminant. Since the coefficient of $Y_0^2$ is positive the denominator has the same sign as $z_0$. The numerator is a term of the form $A(z_0)Y_0+B(z_0)$, $A$ and $B$ positives, multiplied by a positive term, everything preceded by a minus sign.

It follows that the curves $ z = 0 $ and $ AY + B = 0,$ in $Son'$, are the boundaries of $ Son'_s $ and $ Son'_f $. If $Y>-B/A$ and $z>0$ the point is in $Son'_s$, switching according to sign change.
 
 The curve $AY+B=0$ is the projection of sonic' fold curve on  $zY$ plane( the third equation in \eqref{eq:parsonli}). In $Son'$, $z=0$ is the straight line $(t,z=0,Y=-2c)$. It is the intersection of $Son'$ with the surface Sigma, $\Sigma$, formed by Hugoniot curves through points of the secondary bifurcation, $B_0$( see Appendix 1). The straight line and sonic' fold curve intersect transversally forming a saddle point in $Son'$.
 
\end{proof}

\section {\bf Lax Admissibles Regions in $M^3$. \label{sec:lax}}

In this section we identify regions in  $M^3$  where Lax's conditions are satisfied, indicating in which regions there are local shock curve arcs, and in which of these each regions there are nonlocal shock curve arcs. We refer the reader to section 3 of \cite{AEMP10} for admissibility conditions in $M^3$.

Giving a point $\mathcal {U} \in M^3$, we use $sh(\mathcal {U})$ ( $sh'(\mathcal {U})$) to denote the Hugoniot (Hugoniot') curve through $\mathcal {U}$. We define $\mathcal {U}_s$, $\mathcal {U}_f$, $\mathcal {U'}_s$ and $\mathcal {U'}_f$ to be the points

$$
\mathcal {U}_s=sh (\mathcal {U}) \cap \mathcal C_s, \hspace{.2cm}
\mathcal {U}_f=sh (\mathcal {U}) \cap \mathcal C_f,\hspace{.2cm}
\mathcal {U}_{s}^{\prime} = sh'(\mathcal {U}) \cap \mathcal C_s \hspace{.2cm} \text { and } \hspace{.2cm}
\mathcal {U}_{f}^{\prime} = sh'(\mathcal {U}) \cap \mathcal C_f.
$$

We introduce the following simplifying assumptions, \cite{AEMP10}: 

\begin{assump}
\label{assu:assumption1}
We will deal only with points $\mathcal{U}$ in $M^3$ such that the points
$\mathcal {U}_s$ and $\mathcal{U}_f$ always exist.
\end{assump}
\begin{assump}
\label{assu:bifsec}
We consider only Hugoniot and Hugoniot$'$ curves that are 
diffeomorphic to $\mathbb{R}$; of course, they do not 
contain points in the secondary bifurcation locus.
\end{assump}

A Hugoniot curve arc is admissible if it satisfies the following conditions:

\begin{itemize}
\item [{\bf L1 -}] It is oriented in the direction of the speed, $s$, decreasing;
\item [{\bf L2 -}] The speed $s$ in any point $\mathcal U$ in the arc satisfies $s(\mathcal {U})<s( \mathcal {U}_s)$;
\item [{\bf L3 -}] The speed $s$ in any point $\mathcal U$ in the arc satisfies $s(\mathcal {U}_{s}^{\prime})<s(\mathcal {U})<s( \mathcal {U'}_f)$.
\end{itemize}

It follows from Lax conditions that:

\begin{itemize}
\item [1-] Condition L1 implies that admissible arcs do not contain points of $Son$ since the speed $s$ reaches extremum values in $Son$.
\item [2-] Condition L2 implies that admissible arcs begin at $C_s$ or $Son'_s$ since the speed $s$ at a point $\mathcal {U} \in Son'$ is such that $s(\mathcal {U})=s( \mathcal {U'}_s)$.
\item [3-] Conditions L1 and L2 imply that an admissible arc ends in $Son$ or goes to infinity (as $s$ decreases and $z$ goes to $-\infty$). 
\item [4-]  By the continuity of the speed $s$ it is sufficient to check condition L3 in points of $Son'_s$ for admissibles arcs starting there.
\end{itemize}

We can summarize as follow: Admissibles arcs start at $C_s$ or $Son'_s$ oriented toward $s$ decreasing, They end at $Son$ or go to infinity. Since the condition L3 is trivially satisfied at points in $C_s$, \cite{AEMP10}, it is sufficient to check it at initial point in $Son'_s$.
\begin{definition}
Admissible arcs starting at $C_s$ are called local admissible arcs or local shocks. Admissible arcs starting at $Son'_s$ are called non local admissible arcs or non local shocks. In $M^3$ regions formed by admissible arcs are called admissible regions, otherwise they are called non admissible regions. An admissible region formed by local admissible arcs is called a local admissible region and an admissible region formed by non local admissible arcs is called a non local admissible region.  
\end{definition}

We aim to determine the admissible regions in $M^3$. It is well known that a Hugoniot curve intersecting $\mathcal C$ has a local arc, \cite{AEMP10}. Given a point in $\mathcal{C}_s$, we know that the Hugoniot curve is transversal to $\mathcal{C}$ at this point. Let us find out to which side of $\mathcal{C}_s$ does $s$ decreases. This side will be formed by local admissible arcs. Given a point $(t_0,z_0,Y=0)$ in $\mathcal{C}_s$, we have the expression of $s$ restricted to the Hugoniot curve, as a function of $z$. So we will compute its derivative, $\displaystyle \frac{ds}{dz}$ at this point. If it is negative, $s$ is decreasing as $z$ increases, and if it is positive, $s$ decreases as $z$ decreases. How do we know to which side is $z$ increasing? To answer this question we check the $Y$ component of the Hugoniot curve parametrization. If $\displaystyle \frac{dY}{dz}>0$, $z$ increases as we move from the half space $Y>0$ into the half space $Y<0$, and otherwise if $\displaystyle \frac{dY}{dz}<0$.

The speed $s$ along a Hugoniot curve through $(t_0,z_0,0)\in \mathcal{C}$ is given by the equation \eqref{eq:velcar}. The derivative of $s$ at this point is given by 
$$\frac{ds_{ch}}{dz}=\frac{c((b_1+1)z_0^5+(b_1+1)z_0^3+3z_0)t_0+(b_1-1)z_0^2+1)}{(z_0^2+1)((b_1-1)z_0^2+1)}.$$

The numerator of this expression is what we obtain by making $Y=0$ in either of the two equations \eqref{eq:son} and \eqref{eq:son'}. So it is the equation of the inflection locus, $IL$, in the $Y=0$ plane, i.e., $\mathcal{C}$. Since the denominator is positive, the sign of $\displaystyle \frac{ds}{dz}$ changes as we cross $IL$. Since at $(0,0,0)$ it is positive, we can say that $\displaystyle \frac{ds}{dz}>0$ between two branches of $IL$ and is negative outside them.

From the first equation in \eqref{eq:hugch}, the derivative of $Y$ with respect to $z$ at $(t_0,z_0)$ is given by $\frac{dY}{dz}=-\frac{2c(z_0^2+1)t_0}{(b_1+1)z_0^2+1}$. So $\frac{dY}{dz}>0$ in $\mathcal{C}_s$.

We see that $\mathcal{C}_s$, $t<0$, is divided in two subregions by the $t<0$ branch of $IL$. In the subregion of $\mathcal{C}_s$ between the fold, $t=0$, and $IL$, $\displaystyle \frac{dY}{dz}>0$ and $\displaystyle \frac{ds}{dz}>0$. It follows that $s$ increases from under the plane $Y=0$ to over it, so the local admissible arcs are the ones going into the above tunnel region. In the other subregion, the local admissible arcs are the ones going up, i.e., into the lateral region.

Thus we can say that the admissible local regions are respectively contained in the $Y>0$, $t<0$ lateral region, and in the above the tunnel region. The above the tunnel region is divided in three regions by $Tf$: a subregion under $\mathcal{C}_f$, $UC_f$, a subregion under $\mathcal{C}_s$, $UC_s$ and the interior of the surface $Tf$, $ITf$. The last region is formed by Hugoniot curves which do not intersect $\mathcal{C}$. Only subregion $UC_s$ contains local admissible   
arcs.

\begin{remark}
In principle, if an admissible arc intersects $Tf'$, then the condition L3 stops being verified because the Hugoniot' curve through the intersection point does not intersect $\mathcal{C}$. However, in view the above there are no such intersections, since $Tf'$ is contained in a non admissible region, $Y>0$, under the bridge.  
\end{remark}

Let us study non local admissible arcs, i.e., Hugoniot curve arcs which start $Son'_s$, $s$ decreasing. We will proceed as for $\mathcal{C}_s$ and compute $\displaystyle \frac{ds}{dz}>0$ for a point $(t'_0,z_0,Y_0)$ in $Son'$.

Differentiating equation \eqref{eq:speedson'} with respect to $z$ and putting $z=z_0$ we get that the derivative of the speed $s$ at the point $(t'_0,z_0,Y_0)$ is given by
$\frac{ds}{dz}(t'_0,z_0,Y_0)=\frac{Y_0((b_1+1)z_0^2-1)}{(b_1-1)z_0^2-1}$. So it has the same sign of $sss=Y_0((b_1+1)z_0^2-1)$.

To check for which side of $Son'_s$ the speed $s$ is increasing, we will compute the scalar product of the tangent vector of the Hugoniot curve $(t_{hug},z,Y_{hug})$ and the gradient of $Son'$, $\nabla son'$, both evaluated at $(t'_0,z_0,Y_0)$. Since at the origin $son'$ is equal to $-2c$, the vector $\nabla son'$points away from the region above the tunnel, for $Y<0$, or away from the region under the bridge, for $Y>0$. Straight forward computations give the scalar product
 $$\textstyle scal=\textstyle \frac{-2(z_0^2+1)((b_1+1)z_0^2-1)(((b_1+1)^2z_0^4+2(b_1+3)z_0^2+1)Y_0+2c((b_1-1)z_0^2+1))}{((b_1-1)z_0^2+1)((b_1+1)z_0^2+3)z_0}.$$
 
The denominator of $scal$ has the sign of $z_0$. The numerator of $scal$ has the sign of $-((b_1+1)z_0^2-1)(((b_1+1)^2z_0^4+2(b_1+3)z_0^2+1)Y_0+2c((b_1-1)z_0^2+1))$. The second term in this expression is of the form $Ax+B$. From the third equation in the equations \eqref{eq:parsonli} which define the sonic' fold curve, we have $Y=-\frac{B}{A}$. So $scal$ has the sign of $-z_0((b_1+1)z_0^2-1)(Y-(-\frac{B}{A}))$. 

Let us suppose that $(b_1+1)z_0^2-1<0$, in this case $scal$ has the sign of $ z_0(Y-(-\frac{B}{A}))$ which is positive in $Son'_s$, so $z$ increases away from the region under the bridge for $Y>0$, and away from the region above the tunnel, for $Y<0$. It is easy to see that $sss$ has the same sign of $-Y_0$, so for $Y>0$, $sss$ decreases when $z$ increases, i.e., into the lateral. For $Y<0$ $sss$ decreases when $z$ decreases, so, into the region above the tunnel. 

We consider only $(b_1+1)z_0^2-1<0$, since it will be shown that is the condition for $L3$ to be satisfied.

Let us check now condition $L3$. As before, it is sufficient to check it for points in $Son'$. We will take a point $(t_0,z_0,Y_0)$ in $Son'_s$ and the Hugoniot' curve through this point. We will compute the intersections with $\mathcal {C}_s$ and $\mathcal {C}_f$, the values of the speed $s$ at these three points and verify that the inequalities conditions are satisfied, if and only if $(b_1+1)z_0^2-1<0$. This is a long straightforward computation, which was done using the software Maple\textsuperscript{\textregistered }, and we present in Appendix B.

So far we have established that admissible regions are contained in 2 of the regions bounded by $\mathcal{C}$, $Son$, $Son'$ and $Tf$: the lateral region $z>0$, $Y>0$ and $t<0$, and the above the tunnel region under the half plane $Y=0$, $t<0$. 

The lateral region $t<0$, $z>0$, $Y>0$ is bounded by $\mathcal{C}_s$ and $Son'_s$. It contains both local and non local admissible arc. What separates a local arc (starting in $\mathcal{C}_s$) from a non local arc (starting in $Son'_s$)? clearly is the surface formed by Hugoniot curves through points in $IL$. 

It remains to be determined what are the local and non local admissible regions, contained in the region $Y<0$, $t<0$, above the tunnel 

\section*{Appendices\label{sec:Appendices}}

\setcounter{section}{0}
 
\section{The surface Sigma\label{app:sigma}}

In this appendix we will obtain the equation of the surface Sigma. We know that the Hugoniot curves through points of the secondary bifurcation $B$ have two components: a straight line $hugl$ and a curve $hugc$ (see section \ref{sec:wmhc}). The secondary bifurcation is defined by $l=-2c$, \cite{Eschenazi02}, so $hugc$ is obtained as the solution of the system 

\begin{equation*}
\left \{
\begin{array}{l}
(1-Z^2)\tilde{V}-Z\tilde{U}+c=0\\
2\tilde{U}+b_1X=k\\
2\tilde{V}+ZX=-2c,
\end{array}
\right.
\label{eq:hugc}
\end{equation*}

\noindent In this way the surface Sigma is given by 

\begin{equation}
\{\tilde{U}=\frac{(Z^2-1)k-2b_1c}{2(Z^2+b_1-1},\tilde{V}=-\frac{Zk+2c(b_1-1)}{2(Z^2+b_1-1},X=-\frac{2cz-k}{(Z^2+b_1-1}\}
\label{eq:sigma}
\end{equation}
\noindent in the 5 variables $\tilde{U}$, $\tilde{V}$, $Z$, $X$, $k$. Solving the third equation in (\ref{eq:sigma}) for $k$ and substituting the result in the two others equations, we get that the surface Sigma is given by 

\begin{equation}
\{\tilde{U}=\frac{(Z^2-1)X}{2}+Zc, \tilde{V}=-(\frac{ZX}{2}+c)\}
\label{eq:sigma1}
\end{equation}
\noindent in the 4 variables $\tilde{U}$, $\tilde{V}$, $X$, $Z$. 

In order to obtain the surface Sigma in the variables $t$, $z$, $Y$, we use $Z=\displaystyle \frac{1}{z}$ and $X=zY$ in equations \eqref{eq:sigma1} getting:
\begin{equation}
\tilde{U}= \frac{(1-z^2)Y+2c}{2z};
\tilde{V}=-c-\frac{Y}{2}
\label{eq:sigma2}
\end{equation}

\noindent 
Substituting $\tilde{U}$ or $\tilde{V}$ in equation \eqref{eq:sigma2} by the corresponding expression from equation \eqref{eq:tdef}, we get that the surface Sigma is given by $$Y=-\frac{2(tz^3+tz+c)}{z^2+1}.$$

\begin{remark} 
If one is going to use the $\tilde{V}$ expression, one must remember $\tilde{V}=V_1-c$.
\end{remark}

\section{Characterizing the condition L3\label{app:L3}}

In this appendix we will proof that the necessary and sufficient condition to verify the condition L3  is that $z^2<\displaystyle \frac{1}{b_1+1}.$

As in section \ref{sec:Yzt-subsection}, parametric equations for the Hugoniot' curve through a point $(t_0,z_0,Y_0)$ is given by 

\begin{equation}
\left\{
\begin{array}{l}
Y=\displaystyle \frac{A'z^2+B'z +C'}{(z_0^2+1)((b_1-1)z^2+1)}\\
t=\displaystyle \frac{D'z^3+E'z^2+F'z+G'}{c(z_0^2+1)((b_1-1)z^4+b_1z^2+1)},
\end{array}
\right.
\label{eq:hugpontoli}
\end{equation}

\noindent where $A'$, $B'$, $C'$, $D'$, $E'$, $F'$, $G'$ are obtained from equations \eqref{eq:hugponto} by changing $Y$ into $-Y$ and $Y_0$ into $-Y_0$.

Parametric equations for Hugoniot' curve through a point of Son' are obtained from the equations \eqref{eq:hugpontoli} substituting $t_0$ by $t'_0$, $t'_0$ as in section \ref{sec:sonlifs}, getting $(Y_{hug'}(z),t_{hug'}(z))$.



We need to calculate the speed $s$ at the intersections points of the Hugoniot' curve with $\mathcal{C}$ and then write the condition $L3$. As the inequalities involved change only in $Son$, it is enough to verify them in a point of $Son'$.
The  speed $s$ at the point $(t'_0,z_0,Y_0)$ is given by the equation \eqref{eq:speedsonponto'}.


Let $z_{1}$ and $z_{2}$ be, $z_{1}< z_{2}$, the solutions of  the equation $Y_{hug'}=0$. The roots $z_{1}$ and $z_{2}$ are the $z$ coordinates of the intersection points of the Hugoniot' curve with $\mathcal {C}$. The speed, $s_{hug'}(z)$, along the Hugoniot' curve through $(t'_0,z_0,Y_0)$ is obtained by changing $t$ into $t_{hug'}(z)$ in equation \eqref{eq:speed}. We must compare the values $s_{hug'}(z_1)$ and $s_{hug'}(z_2)$ with $s_{son'}$. Denoting by $s_{hug'n}(z)$ the numerator of $s_{hug'}(z)$, $s_{hug'd}(z)$ the denominator of $s_{hug'}(z)$ and $Y_{hug'n}(z)$ the numerator of $Y_{hug'}(z)$, we can write

\begin{equation*}
s_{hug'n}(z)=p_1(z)Y_{hug'n}(z)+r_1(z)
\end{equation*}
and 
\begin{equation*} 
s_{hug'd}(z)=p_2(z)Y_{hug'n}(z)+r_2(z).
\end{equation*}
It follows that, $s_{hug'\mathcal {C}}(z)$, given by
$$s_{hug'\mathcal {C}}(z)= \frac{r_1(z)}{r_2(z)},$$ has the same value of $s_{hug'}(z)$ at the intersection points of the Hugoniot' curve through $(t'_0,z_0,Y_0)$ with $\mathcal {C}$. The expression of $s_{hug'\mathcal {C}}(z)$ is of the form $\displaystyle \frac{az+b}{cz+d}$. Our goal is to get a condition such that $s_{son'}(t'_0,z_0.Y_0)$ satisfies\\ $s_{hug'\mathcal {C}}(z_1)< s_{son'}(t'_0,z_0.Y_0)<s_{hug'\mathcal {C}}(z_2)$. In a more general way we have the following problem: given a number $s$ we want a condition such that 
$$\frac{az_1+b}{cz_1+d}<s<\frac{az_2+b}{cz_2+d},$$ 
where $z_1$ and $z_2$ are the roots of the polynomial $fz^2+gz+h$. Straightforward computations give that the condition is $$\frac{(c^2h-dcg+d^2f)s^2+(agd+bcg-2(ach+bdf))s+a^2h-abg+b^2f}{c^2h-dcg+d^2f}<0.$$

Changing $a$, $b$ into the coefficients of the numerator of $s_{hug'\mathcal {C}}(z)$, $c$ and $d$ into the coefficients of the denominator of $s_{hug'\mathcal {C}}(z)$ and $f$, $g$, $h$ into the coefficients of the numerator of $Y_{hug'}(z)$, we get that the condition is $$cond= \frac{Y_0^2((b_1+1)z_0^2-1}{2}<0,$$
it follows that the condition L3 is satisfied if and only if $\displaystyle \frac{-1}{\sqrt{b_1+1}}<z< \displaystyle \frac{1}{\sqrt{b_1+1}}.$

\newpage 
\bibliographystyle{amsplain} 
\bibliography{paper}

\def\Canic{\v{C}ani\'c}\def\Canic{\v{C}ani\'c}\def\Zoladek{\.Zo\l{}adek}\def\Freistuhler{Freist\"uhler}\def\Freistuhler{Freist\"uhler}\def\Freistuhler{Freist\"uhler}\def\Freistuhler{Freist\"uhler}\def\Freistuhler{Freist\"uhler}\def\Zoladek{\.Zo\l{}adek}
\providecommand{\bysame}{\leavevmode\hbox to3em{\hrulefill}\thinspace}
\providecommand{\MR}{\relax\ifhmode\unskip\space\fi MR }
\providecommand{\MRhref}[2]{%
  \href{http://www.ams.org/mathscinet-getitem?mr=#1}{#2}
}
\providecommand{\href}[2]{#2}
\begin{thebibliography}{1}

\bibitem{AEMP10}
A.~V. Azevedo, C.~S. Eschenazi, D.~Marchesin, and C.~F. Palmeira,
  \emph{Topological resolution of riemann problems for pairs of conservation
  laws}, Quarterly of Applied Mathematics \textbf{68} (2010), 375--393.

\bibitem{Eschenazi02}
C.~S. Eschenazi and C.~F.~B. Palmeira, \emph{The structure of composite
  rarefaction-shock foliations for quadratic systems of conservation laws},
  Matem\'atica Contempor\^anea \textbf{22} (2002), 113--140.

\bibitem{Eschenazi13}
\bysame, \emph{Intersections of hugoniot curves with the sonic surface in the
  wave manifold}, Bulletin of the Brazilian Mathematical Society, New Series
  \textbf{44(2)} (2013), 255--272.

\bibitem{Isaacson92}
E.~Isaacson, D.~Marchesin, C.~F. Palmeira, and B.~Plohr, \emph{A global
  formalism for nonlinear waves in conservation laws}, Comm. Math. Phys.
  \textbf{146} (1992), 505--552.

\bibitem{Marchesin94b}
D.~Marchesin and C.~F.~B. Palmeira, \emph{Topology of elementary waves for
  mixed-type systems of conservation laws}, Journal of Dynamics and
  Differential Equations \textbf{6} (1994), no.~3, 421--440.

\bibitem{Schaeffer87}
D.~Schaeffer and M.~Shearer, \emph{The classification of $2\times2$ systems of
  non-strictly hyperbolic conservation laws, with application to oil recovery,
  with appendix by {D. Marchesin}, {P.J. Paes Leme}, {D.G. Schaeffer}, {M.
  Shearer}}, Comm. Pure Appl. Math. \textbf{40} (1987), 141--178.

\end{thebibliography}

\end{document}